\documentclass[a4paper,11pt]{amsart}
\usepackage{amsmath,amssymb,amsthm,lscape}
\usepackage[dvips,all,arc,curve,color,frame,poly]{xy}
\usepackage[usenames]{color}
\input xy
\xyoption{all}

\numberwithin{equation}{section}
\newtheorem{thm}{Theorem}[section]
\newtheorem{lemma}[thm]{Lemma}
\newtheorem{prop}[thm]{Proposition}
\newtheorem{cor}[thm]{Corollary}

\theoremstyle{definition}
\newtheorem{exam}[thm]{Example}
\newtheorem{define}[thm]{Definition}
\newcommand{\image}[1]{\mathrm{Im}\: #1}                      
\newcommand{\kernel}[1]{\mathrm{Ker}\: #1}                    
                
\newcommand{\homo}[3]{\mathrm{Hom}_{#1}(#2,#3)}               
\newcommand{\enmo}[2]{\mathrm{End}_{#1}(#2)}                  
\newcommand{\aut}[2]{\mathrm{Aut}_{#1}(#2)}                  
                
\newcommand{\socle}[1]{S(#1)}                        
                 
\newcommand{\jac}[1]{J(#1)}                        
\newcommand{\kjac}[2]{J^{#1}(#2)}

\newcommand{\pd}[1]{\mathrm{proj.dim}\: #1}

\newcommand{\leng}[1]{\mathrm{length}(#1)}                 
\newcommand{\ext}[4]{\mathrm{Ext}^{#1}_{#2}(#3,#4)}           
       
\newcommand{\tilt}[1]{\mathrm{tilt}(#1)}                      
\newcommand{\utri}[2]{\mathrm{T}_{#1}(#2)}                   
\newcommand{\fac}[3]{P_{#1,#2,#3}}                          
\newcommand{\lfac}[2]{\overline{P}_{#1,#2}}

\begin{document}
\title{The classification of tilting modules over Harada algebras}
\author{Kota Yamaura}
\address{Graduate School of Mathematics\\ Nagoya University\\ Frocho\\ Chikusaku\\ Nagoya\\ 464-8602\\ Japan}
\email{m07052d@math.nagoya-u.ac.jp}
\date{}
\maketitle

\begin{abstract}
In the 1980s,  Harada introduced a new class of algebras now called Harada algebras .
Harada algebras provides us with a rich source of Auslander's $1$-Gorenstein algebras. 
In this paper, we have two main results about Harada algebras.
The first is the classification of modules over Harada algebras whose projective dimension is at most one. 
The second is the classification of tilting modules over Harada algebras, 
which is shown by giving a bijection between tilting modules over Harada algebras and tilting modules over 
direct products of upper triangular matrix algebras over $K$. 
A combinatorial description of tilting modules over  upper triangular matrix algebras over $K$ is known. 
These facts allow us to classify tilting modules over a given Harada algebra.
\end{abstract}

\section{Main results}
Two classes of algebras have been studied for a long time. 
The first is Nakayama algebras and the second is quasi-Frobenius algebras. 
In the 1980s,  Harada introduced a new class of algebras now called Harada algebras, 
which give a common generalization of Nakayama algebras and quasi-Frobenius algebras. 
Many authors have studied the structure of Harada algebras (e.g. \cite{quasi_H,Lec_N,dua_H,LE_QF,LE_gur,OHr1,OHr2,OHr3}). 
Now let us recall that left Harada algebras are defined from a structural point of view as follows.

\begin{define}\label{def_Harada}
Let $R$ be a basic algebra and $\mathrm{Pi}(R)$ be a complete set of
orthogonal primitive idempotents of $R$ \cite{AF}.
We call $R$ a \emph{left Harada algebra} if $\mathrm{Pi}(R)$ can be arranged such that $\mathrm{Pi}(R)=\{e_{ij}\}_{i=1}^{m},_{j=1}^{n_i}$
where
\begin{enumerate}
\def\labelenumi{(\theenumi)}
\item $e_{i1}R$ is an injective $R$-module for any $i=1,\cdots,m$,
\item $e_{ij}R \simeq e_{i,j-1}\jac{R}$ for any $i=1,\cdots,m$, $j=2,\cdots,n_i$.
\end{enumerate}
Here $\jac{R}$ is the Jacobson radical of $R$.
\end{define}

In \cite{Thr}, Thrall studied three properties of quasi-Frobenius algebras, called QF-$1$, QF-$2$, and QF-$3$. 
It follows from definition that left Harada algebras satisfy the property QF-$3$ which is the condition that the injective hull of the algebra  is projective. 
This property is called $1$-Gorenstein by Auslander (and dominant dimension at least one by Tachikawa) \cite{AR,FGR,HI,IST,T},  and often plays an important role in the representation theory. 
Left Harada algebras form a class of $1$-Gorenstein algebras, and their indecomposable projective modules have "nice" structure.

In this paper, we classify tilting modules over left Harada algebras.
Tilting modules provide us a powerful tool in the representation theory of algebras and are due to \cite{APR,BB,BGP}.

\begin{define}\label{def_tilt}
Let $R$ be an algebra. An $R$-module $T$ is called a \emph{partial tilting module} if $T$ satisfies the following conditions. 
\begin{itemize}
\item[$(1)$] $\pd{T} \leq 1$.
\item[$(2)$] $\ext{1}{R}{T}{T}=0$.
\end{itemize}
Moreover, a partial tilting $R$-module $T$ is called a \emph{tilting module} if $T$ satisfies the following condition.
\begin{itemize}
\item[$(3)$] There exists an exact sequence
\[
0 \longrightarrow R_R \longrightarrow T_0 \longrightarrow T_1 \longrightarrow 0
\]
where $T_0$, $T_1 \in \mathrm{add}T$.
\end{itemize}
\end{define}

We can see from the above definition that tilting modules are a generalization of progenerators which appear in Morita's theorem.
Morita's theorem shows that any progenerator $P$ over an algebra $R$ induces a category equivalence between $\mathrm{mod}R$ and $\mathrm{mod}(\enmo{R}{P})$.
A generalization of Morita's theorem is the Brenner-Butler theorem. It says that any tilting module $T$ over an algebra $R$ induces two category equivalences
between certain full subcategories of $\mathrm{mod}R$ and of $\mathrm{mod}(\enmo{R}{T})$.
As a consequence of the Brenner-Butler theorem, some problems about $R$ can be shifted to those of $\enmo{R}{T}$ (for example, finiteness of global dimension).
By this reason, tilting modules are important for the study of algebras and 
a finding classification of tilting modules over a given algebra is an important problem in representation theory. 
The aim of this paper is to give a classification of tilting modules over a left Harada algebra.

Now we present the two main results of this paper. 
Let $R$ be a left Harada algebra as in Definition \ref{def_Harada}.  
We denote by $\jac{M}$ the Jacobson radical of the $R$-module $M$, by $\kjac{k}{M}$ the $k$-th Jacobson radical of $M$ and by $\socle{M}$ the socle of $M$.
We put
\begin{eqnarray}
P_{ij} := \kjac{j-1}{e_{i1}R} \simeq e_{ij}R \hspace{1cm} (1 \leq i \leq m, \ 1 \leq j \leq n_i) \label{projective}
\end{eqnarray}
for simplicity. 
Then we have a chain 
\[
P_{i1} \supset P_{i2} \supset \cdots \supset P_{in_i}
\]
of indecomposable projective $R$-modules.

The first main result is the classification of $R$-modules whose projective dimension is equal to one. 
The full subcategory of the module category over a $1$-Gorenstein algebra that consists of modules 
whose projective dimension is at most one is contravariantly finite (\cite{HI,IST}). 
Actually, we have a stronger result for left Harada algebra. Namely, 
the full subcategory of the module category over a left Harada algebra that consists of modules whose projective dimension is at most one is finite.
It is obvious that $\pd{(P_{ik}/P_{il})} = 1$. We can show converse.

\begin{thm} \label{classify_pd1}
A complete set of isomorphism classes of indecomposable $R$-modules whose projective dimension is equal to one is given as follows.
\[
\{ P_{ik}/P_{il}  \ | \ 1 \leq i \leq m, \ 1 \leq k < l \leq n_i \}.
\]
\end{thm}

This will be proved in Section $2$ by using key lemmas which follow from Definition \ref{def_Harada} directly.

The other main result is the classification of basic tilting $R$-modules.
We denote by $\tilt{R}$ the set of isomorphism classes of basic tilting $R$-modules and by $\utri{n}{K}$ the $n \times n$ upper triangular matrix algebra over $K$
\[
\left( \begin{array}{ccc}
K & \cdots & K \\
 & \ddots & \vdots \\
0 & & K
\end{array} \right).
\]
The following asserts that tilting $R$-modules are described by tilting modules over a direct product of algebras of the form $\utri{n}{K}$ which gives a typical example of 
a Harada algebra.

\begin{thm}\label{main_thm2}
There exists a bijection
\[
\tilt{R} \xrightarrow{\hspace{1cm}} \tilt{\utri{n_1}{K}} \times \tilt{\utri{n_2}{K}} \times \cdots \times \tilt{\utri{n_m}{K}}.
\]
\end{thm}

We will construct the above correspondence in Section $4$. 
By the well-known classification of tilting modules over upper triangular matrix algebras over $K$ which we recall in Section $5$, 
we can classify tilting modules over a given left Harada algebra by the above correspondence. 

Moreover, we can describe tilting $\utri{n}{K}$-modules combinatorially by using non-crossing partitions of a regular $(n+2)$-polygon.
In particular, we have the following application.

\begin{cor}
The number of basic tilting $R$-modules is equal to
\[
\prod_{i=1}^{m} \frac{1}{n_i+1}  {2n_i \choose n_i}.
\]
\end{cor}

Throughout this paper,  an algebra means a finite dimensional associative algebra over an algebraically closed field $K$.
We always deal with finitely generated right modules over algebras.
\bigskip

\noindent \textit{Acknowledgement}
The author is deeply grateful to Professor Kiyoichi Oshiro for giving me a chance to study ring theory and his warm encouragement. 
The author would like to thank Professor Osamu Iyama for generous support and suggestions getting to the points, and Martin Herschend and Michael Wemyss 
for helpful comments and suggestions.

\bigskip

\section{Proof of Theorem \ref{classify_pd1}}
In this section, we first give some key lemmas of this paper, that is, the properties of homomorphisms between indecomposable projective $R$-modules. 
Next we prove Theorem \ref{classify_pd1} by using these lemmas.

Let $R$ be a left Harada algebra as in Definition \ref{def_Harada}. 
We use the notation \eqref{projective}.

\begin{lemma}\label{basic_lemma}
If a submodule of $P_{i1}$ is not contained in $\jac{P_{in_i}}$, then it is $P_{ij}$ for some $1 \leq j \leq n_i$.
\end{lemma}
\begin{proof}
It follows from Definition \ref{def_Harada} (b).
\end{proof}

\begin{lemma}\label{key_lemma1}
Let $f:P_{ij} \longrightarrow P_{kl}$ be a homomorphism. Then the following hold.
\begin{enumerate}
\def\labelenumi{(\theenumi)}   
\item $f$ is monic if and only if $i=k$, $j \geq l$ and $\image{f}=P_{kj}$.
\item $f$ is not monic if and only if $\image{f} \subset \jac{P_{kn_k}}$.
\item Assume $i=k$ and $j<l$, 
we have $\image{f} \subset \jac{P_{kn_k}}$
\item Assume $i \neq k$, 
we have $\image{f} \subset \jac{P_{kn_k}}$.
\end{enumerate}
\end{lemma}
\begin{proof}
(1) We assume that $f$ is monic.  Then $i=k$ since $\socle{P_{ij}} \simeq \socle{P_{kl}}$ and these are simple. 
By $\leng{P_{ij}} \leq \leng{P_{kl}}$, we have $j \geq l$. 
By Lemma \ref{basic_lemma}, the only submodule of $P_{kl}$ whose length is equal to $\leng{P_{ij}}$ is $P_{kj}$.
The converse follows from $\leng{P_{ij}}=\image{f}$.

(2) We assume that $\image{f} \nsubseteq \jac{P_{kn_k}}$. 
By Lemma \ref{basic_lemma}, there exists $0 \leq r \leq n_k-l$ such that $\image{f}=P_{k,r+l}$. 
Therefore $f$ is monic since $f$ can be seen as an epimorphism between indecomposable projective $R$-modules.
The converse follows from (1).

(3) Since $\leng{P_{ij}} > \leng{P_{il}}$, there exists no monomorphism from $P_{ij}$ to $P_{il}$. 
Therefore the assertion follows from (2). 

(4) Since $i \neq k$, $\socle{P_{ij}}$ and $\socle{P_{kl}}$ are not isomorphic. 
Hence there exists no monomorphism from $P_{ij}$ to $P_{kl}$. 
Therefore the assertion follows from (2).
\end{proof}

\begin{lemma}\label{key_lemma2}
Let $f:P_{ij} \longrightarrow P_{il}$ be any monomorphism with $j \geq l$.
Then the following hold.
\begin{enumerate}
\def\labelenumi{(\theenumi)}  
\item For any homomorphism $g:P_{ij} \longrightarrow P_{il'}$ with $l \geq l'$, 
there exists a homomorphism $h:P_{il} \longrightarrow P_{il'}$ such that $g=hf$.
\item For any homomorphism $g:P_{ij} \longrightarrow P_{st}$ which is not monic, 
there exists a homomorphism $h:P_{il} \longrightarrow P_{st}$ such that $g=hf$.
\item For any homomorphism $g:P_{il'} \longrightarrow P_{il}$ with $l' \geq j$, 
there exists a homomorphism $h:P_{il'} \longrightarrow P_{ij}$ such that $g=fh$.
\item For any homomorphism $g:P_{st} \longrightarrow P_{il}$ which is not monic, 
there exists a homomorphism $h:P_{st} \longrightarrow P_{ij}$ such that $g=fh$.
\end{enumerate}
\end{lemma}
\begin{proof}
(1) Let $u:P_{il'} \longrightarrow P_{i1}$ be the inclusion map. 
Since $P_{i1}$ is injective, there exists a homomorphism $h:P_{il} \longrightarrow P_{i1}$ such that $ug=hf$.
\[
\xymatrix{
0 \ar[r] & P_{ij} \ar[r]^f \ar[d]_g & P_{il} \ar@{.>}[ldd]^h \\
 & P_{il'} \ar[d]_u &  \\
 & P_{i1} &
}
\]
Since $l \geq l'$, we have $\image{h} \subset P_{il'}$. 
We can see $h$ as $h:P_{il} \longrightarrow P_{il'}$.

(2) Let $u:P_{st} \longrightarrow P_{s1}$ be the inclusion map. 
Since $P_{s1}$ is injective, there exists a homomorphism $h:P_{il} \longrightarrow P_{s1}$ such that $ug=hf$.
\[
\xymatrix{
0 \ar[r] & P_{ij} \ar[r]^f \ar[d]_g & P_{il} \ar@{.>}[ldd]^h \\
 & P_{st} \ar[d]_u &  \\
 & P_{s1} &
}
\]
If $h$ is monic, then $ug=hf$ is monic, hence $g$ is monic. This is contradiction. 
Therefore $h$ is not monic. By Lemma \ref{key_lemma1} (2), we have $\image{h} \subset \jac{P_{sn_s}} \subset P_{st}$. 
We can see $h$ as $h:P_{il} \longrightarrow P_{st}$.

(3) By Lemma \ref{key_lemma1} (1),(3), we have $\image{f}=P_{ij}$ and $\image{g} \subset P_{il'}$. 
Since $l' \geq j$, we have $\image{g} \subset \image{f}$. 
Since $P_{il'}$ is projective, there exists a homomorphism $h:P_{il'} \longrightarrow P_{ij}$ such that $g=fh$.
\[
\xymatrix{
& P_{il'} \ar[dd]^g \ar@{.>}[ldd]_h & \\
& & \\
P_{ij} \ar[r]_f & P_{ij} \ar[r] & 0
}
\]

(4) By Lemma \ref{key_lemma1} (1),(2), we have $\image{g} \subset \jac{P_{in_i}} \subset P_{ij}=\image{f}$. 
The assertion follows by the same argument as in the proof of (3). 
\end{proof}

The following result gives Theorem \ref{classify_pd1}. 

\begin{lemma}\label{lemma_pd1} 
Let $Q_i$ and $Q'_j$ be indecomposable projective $R$-modules and 
\[
f: Q:=Q_1 \oplus \cdots \oplus Q_k \longrightarrow Q':=Q'_1 \oplus \cdots \oplus Q'_l
\]
a monomorphism.
Then there exists automorphisms $\varphi \in \aut{R}{Q}$, $\psi \in \aut{R}{Q'}$ such that
\[
\psi f \varphi^{-1}=\left( \begin{array}{ccc}
f_1 & & 0 \\
& \ddots & \\
0 & & f_k \\ \hline
 & 0 & \\
 & \vdots & \\
 & 0 & 
\end{array} \right) : Q_1 \oplus \cdots \oplus Q_k \xrightarrow{\hspace{1cm}} Q'_1 \oplus \cdots \oplus Q'_l .
\]
\end{lemma}
\begin{proof}
We proceed by induction on $k$. 
First we consider the case $k=1$. Then $Q$ is an indecomposable projective $R$-module. 
We write $f$ as 
\[
f:Q \xrightarrow{
\left( \begin{array}{c}
f_1 \\
\vdots \\
f_l
\end{array} \right)
} Q'
, \quad f_i:Q \longrightarrow Q_i \quad (1 \leq i \leq l).
\]
Since $\socle{Q}$ is simple, there exists an monomorphism in $\{f_1,\cdots,f_l\}$. 
So we can assume that $f_1,\cdots,f_r$ are monic and $f_{r+1},\cdots,f_l$ are not monic. 
We assume that $\leng{Q'_1} \leq \leng{Q'_i}$ for $2 \leq i \leq r$. 
Then for any $2 \leq j \leq r$ there exists a homomorphism $h_j:Q'_1 \longrightarrow Q'_j$ such that 
$f_j=h_jf_1$ by Lemma \ref{key_lemma2} (1). 
Moreover, for any $r+1 \leq j \leq n$ there exists a homomorphism $h_j:Q'_1 \longrightarrow Q'_j$ such that 
$f_j=h_jf_1$ by Lemma \ref{key_lemma2} (2). 
Let
\[
\psi=\left( \begin{array}{c|ccc}
1 & 0 & \cdots & 0 \\ \hline
-h_2 & 1 &  & 0 \\
\vdots & & \ddots & \\
-h_l & 0 & & 1
\end{array} \right) \in \aut{R}{Q'_1 \oplus \cdots \oplus Q'_l} .
\]
Then we have
\[
\psi f =
\left( \begin{array}{c|ccc}
1 & 0 & \cdots & 0 \\ \hline
-h_2 & 1 &  & 0 \\
\vdots & & \ddots & \\
-h_l & 0 & & 1
\end{array} \right)
\left( \begin{array}{c}
f_1 \\
\vdots \\
f_l
\end{array} \right)
=\left( \begin{array}{c}
f_1 \\
0 \\
\vdots \\
0
\end{array} \right).
\]

Next we assume that $k \geq 2$ and that the assertion holds for $k-1$. 
We assume that $\leng{Q_k} \leq \leng{Q_i}$ for $1 \leq i \leq k-1$. 
By applying the induction hypotheses to $f|_{Q_1\oplus \cdots \oplus Q_{k-1}}$, we can assume that
\[
f|_{Q_1\oplus \cdots \oplus Q_{k-1}}: Q_1\oplus \cdots \oplus Q_{k-1} \xrightarrow{
\left( \begin{array}{ccc}
f_1 & & 0 \\
& \ddots & \\
0 & & f_{k-1} \\ \hline
& 0 & \\
& \vdots & \\
& 0 &
\end{array} \right)
} Q' , \quad
f_i : Q_i \longrightarrow Q'_i \quad (1 \leq i \leq k-1).
\]
Therefore we can write $f$ as
\[
f:Q \xrightarrow{
\left( \begin{array}{ccc|c}
f_1 & & 0 & g_1 \\
& \ddots & & \vdots \\
0 & & f_{k-1} & g_{k-1} \\ \hline
0 & \cdots & 0 &  g_k \\
\vdots & & \vdots &  \vdots \\
0 & \cdots & 0 &  g_l \\
\end{array} \right)
} Q' , \quad 
g_i:Q_k \longrightarrow Q'_i \quad (1 \leq i \leq l).
\]
By Lemma \ref{key_lemma2} (3),(4), and the assumption on $Q_k$, for any $1 \leq i \leq k-1$ 
there exists a homomorphism $h_i:Q_k \longrightarrow Q_i$ such that $g_i=f_ih_i$. 
Let 
\[
\varphi=\left( \begin{array}{ccc|c}
1 & & 0 & h_1 \\
& \ddots &  & \vdots \\
0 &  & 1 & h_{k-1} \\ \hline
0 & & & 1
\end{array} \right) \in \aut{}{Q_1 \oplus \cdots \oplus Q_k}.
\]
Then we have
\[
\left( \begin{array}{ccc|c}
f_1 & & 0 &  \\
 & \ddots & & 0 \\
0 & & f_{k-1} & \\ \hline
0 & \cdots & 0 &  g_k \\
\vdots & & \vdots &  \vdots \\
0 & \cdots & 0 &  g_l \\
\end{array} \right) \varphi
= \left( \begin{array}{ccc|c}
f_1 & & 0 & g_1 \\
& \ddots & & \vdots \\
0 & & f_{k-1} & g_{k-1} \\ \hline
0 & \cdots & 0 &  g_k \\
\vdots & & \vdots &  \vdots \\
0 & \cdots & 0 &  g_l \\
\end{array} \right)
= f .
\]
By applying the same argument as in the case $k=1$ to
\[
\left( \begin{array}{c}
g_k \\
\vdots \\
g_l
\end{array} \right),
\]
the assertion follows.
\end{proof}

Now we can prove Theorem \ref{classify_pd1}.

\begin{proof}
The projective dimension of $P_{ik}/P_{il}$ is obviously equal to $1$.  
Let $X$ be an indecomposable $R$-module whose projective dimension is equal to one. 
Then there exists an exact sequence
\[
0 \longrightarrow Q \longrightarrow Q' \longrightarrow X \longrightarrow 0
\]
such that $Q$ and $Q'$ are projective $R$-modules. 
By Lemma \ref{lemma_pd1} and since $X$ is an indecomposable $R$-module, $Q$ and $Q'$ must be indecomposable $R$-modules. 
By Lemma \ref{key_lemma1} (1), $X$ is isomorphic to one of $P_{ik}/P_{il}$. 
\end{proof}

\bigskip

\section{Triangular factor algebras of Harada algebras}
In this section, we keep the notations from the previous section. 
We define a special factor algebra $\overline{R}=R/I$ of $R$ which is isomorphic to a direct product $\utri{n_1}{K} \times \cdots \times \utri{n_m}{K}$ 
of upper triangular matrix algebras over $K$. 
The algebra $\overline{R}$ contains important information about $R$ which is seen in Lemma \ref{fun_R/X} and Proposition \ref{prop_ext}. 

We start by giving the following ideal $I$ of $R$.  We put
\[
e_{ij}R \supset I_{ij}:=\kjac{n_i-j+1}{e_{ij}R} \hspace{1cm} (1 \leq i \leq m, \ 1 \leq j \leq n_i),
\]
and
\[
R \supset I:= \bigoplus_{i=1}^{m} \bigoplus_{j=1}^{n_i} I_{ij}.
\]
Then we have the following result.

\begin{lemma}\label{I_ideal} 
$I$ is an ideal of $R$.
\end{lemma}
\begin{proof}
Clearly $I$ is a right ideal of $R$. We show $I$ is a left ideal of $R$. 
It is enough to show that $rx \in I_{kl}=\kjac{n_k-l+1}{e_{kl}R}$ for any $x \in I_{ij}=\kjac{n_i-j+1}{e_{ij}R}$ and any $r \in e_{kl}R$. 
We consider the homomorphism
\[
\varphi_r: I_{ij} \ni x \longmapsto rx \in e_{kl}R
\]
of right $R$-modules. Since $I_{ij}$ is indecomposable and non-projective, we have $\image{\varphi_r} \subset \kjac{n_k-l+1}{e_{kl}R}=I_{kl}$. 
Therefore $I$ is a left ideal of $R$.
\end{proof}

By Lemma \ref{I_ideal}, we can consider the factor algebra 
\[
\overline{R}:=R/I.
\]
We put 
\[
e_i:=e_{i1}+e_{i2}+\cdots+e_{in_i}
\]
for $1 \leq i \leq m$. 

Now we show the following description of $\overline{R}$.

\begin{prop}\label{factor_R/X}
Under the hypotheses above, the following assertions hold.
\begin{enumerate}
\def\labelenumi{(\theenumi)}    
\item  $\{ e_i \ | \ 1 \leq i \leq m \}$ is a set of orthogonal central idempotents of $\overline{R}$ and there exists a $K$-algebra isomorphism 
\[
\overline{R}e_j \simeq \utri{n_i}{K}.
\]
\item  There exists a $K$-algebra isomorphism
\[
\overline{R} \simeq \utri{n_1}{K} \times \utri{n_2}{K} \times \cdots \times \utri{n_m}{K}.
\]
\end{enumerate}
\end{prop}

To prove the above proposition, we describe all indecomposable projective $\overline{R}$-modules as factor modules of indecomposable projective $R$-modules.
Since $I \subset \jac{R}$, we have that
\[
\{e_{ij}+I \in \overline{R}\ | \ 1 \leq i \leq m, \ 1 \leq j \leq n_i \}
\]
is a complete set of orthogonal primitive idempotents of $\overline{R}$. By the $\overline{R}$-module isomorphism 
\[
e_{ij}\overline{R} = e_{ij}R/\kjac{n_i-j+1}{e_{ij}R} \simeq P_{ij}/\kjac{n_i-j+1}{P_{ij}}=P_{ij}/\jac{P_{in_i}},
\]
 a complete set of indecomposable projective $\overline{R}$-modules is
\[
\{ P_{ij}/\jac{P_{in_i}} \ | \ 1 \leq i \leq m, \ 1 \leq j \leq n_i \}.
\]
By Definition \ref{def_Harada} (b),
\begin{eqnarray}
0 \subset P_{in_i}/\jac{P_{in_i}} \subset P_{i,n_i-1}/\jac{P_{in_i}} \subset \cdots \subset P_{i,j+1}/\jac{P_{in_i}} \subset P_{ij}/\jac{P_{in_i}}  \label{cs}
\end{eqnarray}
is a unique composition series of $P_{ij}/\jac{P_{in_i}}$ as an $\overline{R}$-module. 
Therefore any indecomposable projective $\overline{R}$-module is serial and its composition factors are not isomorphic to each other.

\bigskip

From the above argument, we can prove Proposition \ref{factor_R/X}.

\begin{proof}
(1) We calculate $\homo{\overline{R}}{\lfac{i}{j}}{\lfac{k}{l}}$. 
If $i \neq k$, $\lfac{i}{j}$ and $\lfac{k}{l}$ have no common composition factors. So we have
\[
\homo{\overline{R}}{\lfac{i}{j}}{\lfac{k}{l}}=0.
\]
If $i = k$, we can easily see that
\[
\homo{\overline{R}}{\lfac{i}{j}}{\lfac{i}{l}} \simeq \begin{cases}
K & (j \geq l) \\
0 & (j < l) 
\end{cases}
\]
by composition series \eqref{cs}.

Thus we have the follwoing isomorphisms as $K$-vector space.
\begin{eqnarray*}
e_i\overline{R}e_j \simeq \homo{\overline{R}}{e_j\overline{R}}{e_i\overline{R}} 
& \simeq & \left( \begin{array}{cccc}
\homo{\overline{R}}{\lfac{j}{1}}{\lfac{i}{1}} & \homo{\overline{R}}{\lfac{j}{2}}{\lfac{i}{1}} & \cdots & \homo{\overline{R}}{\lfac{j}{n_j}}{\lfac{i}{1}} \\
\homo{\overline{R}}{\lfac{j}{1}}{\lfac{i}{2}} & \homo{\overline{R}}{\lfac{j}{2}}{\lfac{i}{2}} & \cdots & \homo{\overline{R}}{\lfac{j}{n_j}}{\lfac{i}{2}} \\
\vdots & \vdots & \ddots & \vdots \\
\homo{\overline{R}}{\lfac{j}{1}}{\lfac{i}{n_i}} & \homo{\overline{R}}{\lfac{j}{2}}{\lfac{i}{n_i}} & \cdots & \homo{\overline{R}}{\lfac{j}{n_j}}{\lfac{i}{n_i}} \\
\end{array} \right) \\
& \simeq & \begin{cases} 
\left( \begin{array}{cccc}
K & K & \cdots & K \\
 & K & \cdots & K \\
 & & \ddots & \vdots \\
0 & & & K 
\end{array} \right) & (i = j) \\
0 & (i \neq j).
\end{cases}
\end{eqnarray*}
It is easily seen that the above isomorphism gives a $K$-algebra isomorphism when $i=j$.

(2) By (1), we have the following $K$-algebra isomorphism.
\[
\overline{R} \simeq \left( \begin{array}{cccc}
e_1\overline{R}e_1 & e_1\overline{R}e_2 & \cdots & e_1\overline{R}e_m \\
e_2\overline{R}e_1 & e_2\overline{R}e_2 & \cdots & e_2\overline{R}e_m \\
\vdots & \vdots & \ddots & \vdots \\
e_m\overline{R}e_1 & e_m\overline{R}e_2 & \cdots & e_m\overline{R}e_m 
\end{array} \right)
\simeq \left( \begin{array}{cccc}
\utri{n_1}{K} &  &   & 0 \\
 & \utri{n_2}{K} &  &  \\
 &  & \ddots &  \\
0 &  &  & \utri{n_m}{K} 
\end{array} \right).
\]
\end{proof}
\bigskip

\section{Proof of Theorem \ref{main_thm2}}
In this section, we keep the notations from the previous section. 
The aim of this section is to prove Theorem \ref{main_thm2}. 
First we describe the indecomposable $\overline{R}$-modules by using the indecomposable projective $R$-modules $P_{ij}$.
Next we consider a certain functor $F$ preserving the vanishing property $\ext{1}{}{-}{-}=0$, from the category $\mathcal{P}$ of $R$-modules whose projective dimension is at most one 
to $\mathrm{mod}\overline{R}$. 
We construct a bijection between $\tilt{R}$ and $\tilt{\overline{R}}$ by using the functor $F$ which gives Theorem \ref{main_thm2}.

We start by giving a classification of indecomposable $\overline{R}$-modules.
By Proposition \ref{factor_R/X}, it is shown that $\overline{R}$ is a Nakayama algebra.  
It is well-known that any indecomposable module over a Nakayama algebra is isomorphic to some subfactor of an indecomposable projective module (\cite{ASS}). 
By this fact and the unique composition series \eqref{cs} of $P_{ij}/\jac{P_{in_i}}$, 
\[
\{ P_{ik}/P_{il} \ | \ 1 \leq i \leq m, \ 1 \leq k < l \leq n_i \}
\]
is a complete set of indecomposable nonprojective $\overline{R}$-modules.

We put
\[
\lfac{i}{j}:=P_{ij}/\jac{P_{in_i}} \simeq e_{ij}\overline{R}
\]
for $1 \leq i \leq m$, $1 \leq j \leq n_i$ and 
\[
\fac{i}{k}{l}:=P_{ik}/P_{il}
\]
for $1 \leq i \leq m$, $1 \leq k < l \leq n_i$ for simplicity. 
These $R$-modules can be regarded as $\overline{R}$-modules. 

We have the following diagram for any $1 \leq i \leq m$ by our definitions.
\[
\begin{array}{ccccccccccc}
\lfac{i}{1}  & \longrightarrow & \fac{i}{1}{n_i} & \longrightarrow & \fac{i}{1}{n_i-1} & \longrightarrow & \cdots & \longrightarrow & \fac{i}{1}{3} & \longrightarrow & \fac{i}{1}{2} \\
\cup & & \cup & & \cup & & & & \cup & & \\
\lfac{i}{2}  & \longrightarrow & \fac{i}{2}{n_i} & \longrightarrow & \fac{i}{2}{n_i-1} & \longrightarrow & \cdots & \longrightarrow & \fac{i}{2}{3} &  & \\
\cup & & \cup & & \cup & & & &  & & \\
\vdots & & \vdots & & \vdots & & & &  & & \\
\cup  & & \cup & & \cup  & & & & & & \\
 \lfac{i}{n_i-2} & \longrightarrow & \fac{i}{n_i-2}{n_i} & \longrightarrow & \fac{i}{n_i-2}{n_i-1} & & & & & & \\
\cup  & & \cup & & & & & & & & \\
 \lfac{i}{n_i-1} & \longrightarrow & \fac{i}{n_i-1}{n_i} & & & & & & & & \\
\cup  & & & & & & & & & & \\
 \lfac{i}{n_i} & & & & & & & & & & 
\end{array}
\]
In the above diagram, right arrows mean natural epimorphisms. 
We remark that the above diagram is the AR-quiver of $\mathrm{mod}(\overline{R}e_i)$ (see Section $5$).
\bigskip

Let $\mathcal{P}$ be the category of $R$-modules whose projective dimension  is at most one.
We define the full subcategories $\mathcal{P}_i$ of $\mathcal{P}$ for $1 \leq i \leq m$ by
\[
\mathcal{P}_i:=\mathrm{add}\{ P_{ij}, \ \fac{i}{k}{l} \ | \ 1 \leq j \leq n_i, \ 1 \leq k < l \leq n_i \}.
\]
By Theorem \ref{classify_pd1}, we have 
\[
\mathcal{P}=\mathrm{add}(\mathcal{P}_1 \cup \mathcal{P}_2 \cup \cdots \cup \mathcal{P}_m).
\]

A key role is played by the functor
\[
F:=-\otimes_R \overline{R} :\mathcal{P} \xrightarrow{\hspace{1cm}} \mathrm{mod}\overline{R}.
\]

\begin{lemma}\label{fun_R/X}
Under the hypotheses above, the following hold.
\begin{enumerate}
\def\labelenumi{(\theenumi)}   
\item The functor $F$ induces a bijection between the isomorphism classes of $R$-modules which lie in $\mathcal{P}$ and 
the isomorphism classes of $\overline{R}$-modules.
\item The restriction on $F$ to $\mathcal{P}_i$ induces a one to one correspondecne between
the isomorphism classes of $\mathcal{P}_i$ and the isomorphism classes of $\mathrm{mod}(\overline{R}e_i)$.
\end{enumerate}
\end{lemma}
\begin{proof}
We calculate $F(M)$ for an indecomposable $R$-module $M$ which lies in $\mathcal{P}$. 
We have isomorphisms
\begin{eqnarray*}
F(P_{ij}) = P_{ij} \otimes_R \overline{R} \simeq P_{ij}/(P_{ij} I)  = P_{ij}/\kjac{n_i}{P_{i1}} = \overline{P}_{ij},
\end{eqnarray*}
\begin{eqnarray*}
F(\fac{i}{k}{l}) = \fac{i}{k}{l} \otimes_R \overline{R} \simeq \fac{i}{k}{l}/(\fac{i}{k}{l} I) = \fac{i}{k}{l}.
\end{eqnarray*}
for $1 \leq i \leq m$, $1 \leq j \leq n_i$ and $1 \leq k < l \leq n_i$.
The assertion follows.
\end{proof}

Now we state a theorem that gives a bijection between $\tilt{R}$ and $\tilt{\overline{R}}$ by using the functor $F$.

\begin{thm}\label{main_thm}
Under the hypotheses above, every tilting $R$-module lies in $\mathcal{P}$ and we have a bijection
\[
F:\tilt{R} \ni T \longmapsto F(T) \in \tilt{\overline{R}}.
\]
\end{thm}

As a consequence of Theorem \ref{main_thm}, we have the following corollary immediately. 

\begin{cor}\label{main_cor}
Under the hypotheses  above, we have a bijection
\[
\tilt{R} \ni T \longmapsto (F(T)e_1,\cdots,F(T)e_m) \in \tilt{\overline{R}e_1} \times \cdots \times \tilt{\overline{R}e_m}.
\]
\end{cor}

Hence by Proposition \ref{factor_R/X}, we have Theorem \ref{main_thm2}.
In the rest of this section, we prove of Theorem \ref{main_thm} .
\bigskip

We have to know when $\ext{1}{R}{X}{Y}=0$ holds for $X,Y \in \mathcal{P}$.
We start with the following result for $\ext{1}{R}{\mathcal{P}_i}{\mathcal{P}_u}$.

\begin{lemma}\label{ext_1}
$\ext{1}{R}{\mathcal{P}_i}{\mathcal{P}_u}=0$ if $i \neq u$. 
\end{lemma}
\begin{proof}
It is obvious that $\ext{1}{R}{P_{ij}}{\mathcal{P}_u}=0$. 
We show $\ext{1}{R}{\fac{i}{k}{l}}{\mathcal{P}_u}=0$ for $1 \leq k < l \leq n_i$. 

First we show $\ext{1}{R}{\fac{i}{k}{l}}{P_{uv}}=0$ for $1 \leq v \leq n_u$. 
We take a projective resolution
\begin{eqnarray}
0 \longrightarrow P_{il} \xrightarrow{\ f \ } P_{ik} \longrightarrow \fac{i}{k}{l} \longrightarrow 0 \label{proj1}
\end{eqnarray}
of $\fac{i}{k}{l}$ in $\mathrm{mod}R$. By applying $\homo{R}{-}{P_{uv}}$ to the above exact  sequence, 
we have an exact sequence
\[
\homo{R}{P_{ik}}{P_{uv}} \xrightarrow{\homo{}{f}{P_{uv}}} 
\homo{R}{P_{il}}{P_{uv}} \longrightarrow \ext{1}{R}{\fac{i}{k}{l}}{P_{uv}} \longrightarrow 0 .
\]
By the assumption $i \neq u$, there is no monomorphism from $P_{il}$ to $P_{uv}$ since the simple socles $\socle{P_{il}}$ and $\socle{P_{uv}}$ are not isomorphic. 
By Lemma \ref{key_lemma2} (2), $\homo{}{f}{P_{uv}}$ is an epimorphism. 
Therefore we have $\ext{1}{R}{\fac{i}{k}{l}}{P_{uv}}=0$.

Next we show $\ext{1}{R}{\fac{i}{k}{l}}{\fac{u}{s}{t}}=0$ for $1 \leq s < t \leq n_u$.
By applying $\homo{R}{-}{\fac{u}{s}{t}}$ to \eqref{proj1}, we have an exact sequence
\[
\homo{R}{P_{ik}}{\fac{u}{s}{t}} \longrightarrow 
\homo{R}{P_{il}}{\fac{u}{s}{t}} \longrightarrow \ext{1}{R}{\fac{i}{k}{l}}{\fac{u}{s}{t}} \longrightarrow 0.
\]
By the assumption $i \neq u$, $P_{il}/\jac{P_{il}}$ does not appear in composition factors of $\fac{u}{s}{t}$. 
Therefore we have $\homo{R}{P_{il}}{\fac{u}{s}{t}}=0$, and so $\ext{1}{R}{\fac{i}{k}{l}}{\fac{u}{s}{t}}=0$.
\end{proof}

Next we consider $\ext{1}{R}{\mathcal{P}_i}{\mathcal{P}_i}=0$.
We need the following result.

\begin{lemma}\label{lemma_ext}
For any $1 \leq i \leq m$, $1 \leq j \leq n_i$ and $\ 1 \leq k < l \leq n_i$, the natural epimorphism $\varphi:P_{ij} \longrightarrow \lfac{i}{j}$ induces an isomorphism
\[
\homo{}{\varphi}{\fac{i}{k}{l}} : \homo{\overline{R}}{\lfac{i}{j}}{\fac{i}{k}{l}} = \homo{R}{\lfac{i}{j}}{\fac{i}{k}{l}} \longrightarrow \homo{R}{P_{ij}}{\fac{i}{k}{l}} .
\]
\end{lemma}
\begin{proof}
It is obvious that $\homo{R}{\lfac{i}{j}}{\fac{i}{k}{l}} = \homo{\overline{R}}{\lfac{i}{j}}{\fac{i}{k}{l}}$ holds. 
We show that $\homo{}{\varphi}{\fac{i}{k}{l}}$ is an isomorphism. 

Since $\varphi$ is epic, we have that $\homo{}{\varphi}{\fac{i}{k}{l}}$ is monic. 
Since any $f \in \homo{R}{P_{ij}}{\fac{i}{k}{l}}$ satisfies $\kernel{f} \supset P_{ij} I = \kernel{\varphi}$, we have that $f$ factors through $\varphi$. 
Thus $\homo{}{\varphi}{\fac{i}{k}{l}}$ is an isomorphism.
\end{proof}

\begin{prop}\label{prop_ext1}
The following hold.
\begin{enumerate}
\def\labelenumi{(\theenumi)} 
\item  For $1 \leq j \leq n_i$, $\ext{1}{R}{P_{ij}}{\mathcal{P}}=0=\ext{1}{\overline{R}}{\lfac{i}{j}}{\mathrm{mod}\overline{R}}$.
\item For $1 \leq k < l \leq n_i$, $1 \leq s < t \leq n_i$, there is a $K$-vector space isomorphism
\[
\ext{1}{R}{\fac{i}{k}{l}}{\fac{i}{s}{t}} \simeq \ext{1}{\overline{R}}{\fac{i}{k}{l}}{\fac{i}{s}{t}}.
\]
\item For $1 \leq k < l \leq n_i$, $1 \leq j \leq n_i$, 
$\ext{1}{R}{\fac{i}{k}{l}}{P_{ij}}=0$ if and only if $\ext{1}{\overline{R}}{\fac{i}{k}{l}}{\lfac{i}{j}}=0$.
\end{enumerate}
\end{prop}
\begin{proof}
(1) Obvious. 

(2) We have a natural projective resolution
\begin{eqnarray}
0 \longrightarrow P_{il} \xrightarrow{\ f \ } P_{ik} \longrightarrow \fac{i}{k}{l} \longrightarrow 0 \label{proj2}
\end{eqnarray}
of $\fac{i}{k}{l}$ in $\mathrm{mod}R$ and a natural projective resolution 
\begin{eqnarray}
0 \longrightarrow \lfac{i}{l} \xrightarrow{\ f' \ } \lfac{i}{k} \longrightarrow \fac{i}{k}{l} \longrightarrow 0 \label{proj3}
\end{eqnarray}
of $\fac{i}{k}{l}$ in $\mathrm{mod}\overline{R}$. 
For natural epimorphisms $\varphi:P_{ik} \longrightarrow \lfac{i}{k}$ and $\varphi':P_{il} \longrightarrow \lfac{i}{l}$, we have the following commutative diagram.
\[
\xymatrix{
0 \ar[r] & P_{il} \ar[r]^f \ar[d]_{\varphi'} & P_{ik} \ar[r] \ar[d]^{\varphi} & \fac{i}{k}{l} \ar[r] & 0 \\
0 \ar[r] & \lfac{i}{l} \ar[r]^{f'} & \lfac{i}{k} \ar[r] & \fac{i}{k}{l} \ar[r] & 0  
}
\]
By applying $\homo{R}{-}{\fac{i}{s}{t}}$ to the upper row and applying $\homo{\overline{R}}{-}{\fac{i}{s}{t}}$ to the lower row, 
we have the following commutative diagram.
\[
\xymatrix{
\homo{\overline{R}}{\lfac{i}{k}}{\fac{i}{s}{t}} \ar[r] \ar@{=}[d] & \homo{\overline{R}}{\lfac{i}{l}}{\fac{i}{s}{t}} \ar[r] \ar@{=}[d] & 
\ext{1}{\overline{R}}{\fac{i}{k}{l}}{\fac{i}{s}{t}} \ar[r] & 0 \\
\homo{R}{\lfac{i}{k}}{\fac{i}{s}{t}} \ar[d]_{\homo{}{\varphi}{\fac{i}{s}{t}}} & \homo{R}{\lfac{i}{l}}{\fac{i}{s}{t}} \ar[d]^{\homo{}{\varphi'}{\fac{i}{s}{t}}} &  \\
\homo{R}{P_{ik}}{\fac{i}{s}{t}} \ar[r] & \homo{R}{P_{il}}{\fac{i}{s}{t}} \ar[r] & \ext{1}{R}{\fac{i}{k}{l}}{\fac{i}{s}{t}} \ar[r] & 0 
}
\]
By Lemma \ref{lemma_ext}, $\homo{}{\varphi}{\fac{i}{s}{t}}$ and $\homo{}{\varphi'}{\fac{i}{s}{t}}$ are isomorphisms. 
Consequently we have an isomorphism $\ext{1}{R}{\fac{i}{k}{l}}{\fac{i}{s}{t}} \simeq \ext{1}{\overline{R}}{\fac{i}{k}{l}}{\fac{i}{s}{t}}$.

(3) By applying $\homo{R}{-}{P_{ij}}$ to \eqref{proj2}, we have an exact sequence
\[
\homo{R}{P_{ik}}{P_{ij}} \xrightarrow{\homo{}{f}{P_{ij}}} \homo{R}{P_{il}}{P_{ij}} \longrightarrow 
\ext{1}{R}{\fac{i}{k}{l}}{P_{ij}} \longrightarrow 0.
\]
It can be seen that $\ext{1}{R}{\fac{i}{k}{l}}{P_{ij}}=0$ if and only if $\homo{}{f}{P_{ij}}$ is an epimorphism.

We show that $\homo{}{f}{P_{ij}}$ is an epimorphism if and only if $j \leq k$ or $l<j$. 
First we assume that $j>k$ and $l \geq j$. 
By $l \geq j$, there exists a monomorphism from $P_{il}$ to $P_{ij}$. 
But there are no monomorhisms from $P_{ik}$ to $P_{ij}$ by $j>k$. 
Since $P_{ik}$ has simple socle, $gf$ is not monic for any $g \in \homo{R}{P_{ik}}{P_{ij}}$.
Thus $\homo{}{f}{P_{ij}}$ is not an epimorphism. 
Next we assume $j \leq k$. 
By Lemma \ref{key_lemma2} (1), $\homo{}{f}{P_{ij}}$ is an epimorphism. 
Finally we assume $l<j$. 
Then by $\leng{P_{il}} > \leng{P_{ij}}$, there are no monomorphisms from $P_{il}$ to $P_{ij}$. 
By Lemma \ref{key_lemma2} (2), $\homo{}{f}{P_{ij}}$ is an epimorphism. 

On the other hand, by applying $\homo{\overline{R}}{-}{\lfac{i}{j}}$ to \eqref{proj3}, we have an exact sequence
\[
\homo{\overline{R}}{\lfac{i}{k}}{\lfac{i}{j}} \xrightarrow{\homo{}{f'}{\lfac{i}{j}}} \homo{\overline{R}}{\lfac{i}{l}}{\lfac{i}{j}} \longrightarrow 
\ext{1}{\overline{R}}{\fac{i}{k}{l}}{\lfac{i}{j}} \longrightarrow 0.
\]
It can be seen that $\ext{1}{\overline{R}}{\fac{i}{k}{l}}{\lfac{i}{j}}=0$ if and only if $\homo{}{f'}{\lfac{i}{j}}$ is an epimorphism.
We can show that $\homo{}{f'}{\lfac{i}{j}}$ is an epimorphism if and only if $j \leq k$ or $l<j$ by the same argument as above. 

Consequently $\ext{1}{R}{\fac{i}{k}{l}}{P_{ij}}=0$ if and only if $j \leq k$ or $l<j$ 
which is equivalent to $\ext{1}{\overline{R}}{\fac{i}{k}{l}}{\lfac{i}{j}}=0$.
\end{proof}

By the Proposition \ref{prop_ext1}, we have the following proposition.

\begin{prop}\label{prop_ext}
For any $T \in \mathcal{P}$, we have that $\ext{1}{R}{T}{T}=0$ if and only if $\ext{1}{\overline{R}}{F(T)}{F(T)}=0$.
\end{prop}
\begin{proof}
The assertion follows from Lemma \ref{ext_1} and Proposition \ref{prop_ext1}.
\end{proof}

We need the following well-known proposition which describes a very useful equivalent condition of tilting modules. 

\begin{prop}[\cite{ASS}]\label{equ_tilt} 
Let $R$ be an algebra and $T$ a partial tilting $R$-module.
Then the following are equivalent. 
\begin{enumerate}
\def\labelenumi{(\theenumi)}   
\item $T$ is a tilting module.
\item The number of pairwise nonisomorphic indecomposable direct summand of $T$ is equal to 
the number of pairwise nonisomorphic simple $R$-modules.
\end{enumerate}
\end{prop}

Now we can prove Theorem \ref{main_thm}.

\begin{proof}
First we take a basic tilting $R$-module $T$. 
Then $T$ lies in $\mathcal{P}$ by Theorem \ref{classify_pd1}. So we can consider $F(T)$.
By Lemma \ref{fun_R/X}, we have that $F$ gives a bijection between isomorphism classes of $R$-modules whose projective dimension is at most one 
and those of $\overline{R}$-modules.
By Proposition \ref{equ_tilt} and the fact that the number of isomorphism classes of simple $R$-modules is equal to that of  $\overline{R}$-modules, 
we only have to show that $T \in \mathcal{P}$ satisfies $\ext{1}{R}{T}{T}=0$ if and only if $\ext{1}{\overline{R}}{F(T)}{F(T)}=0$. 
This follows from Proposition \ref{prop_ext}.
\end{proof}

\section{Combinatorial description of tilting $\utri{n}{K}$-modules}
First we recall the classification of basic tilting modules over the upper triangular matrix algebra $\utri{n}{K}$. 
Our classification should be well-known for experts \cite{TA,CCS,I,R}. 
Nevertheless we will give a complete proof since there does not seem to exist proper reference.
Next we explain how to classify basic tilting modules over a given left Harada algebra. 

We show the well-known classification of basic tilting $\utri{n}{K}$-modules 
by constructing a bijection between $\tilt{\utri{n}{K}}$ and the set of non-crossing partitions of the regular $n+2$-polygon into triangles. 

First we introduce coordinates in the AR-quiver of $\utri{n}{K}$ as follows.
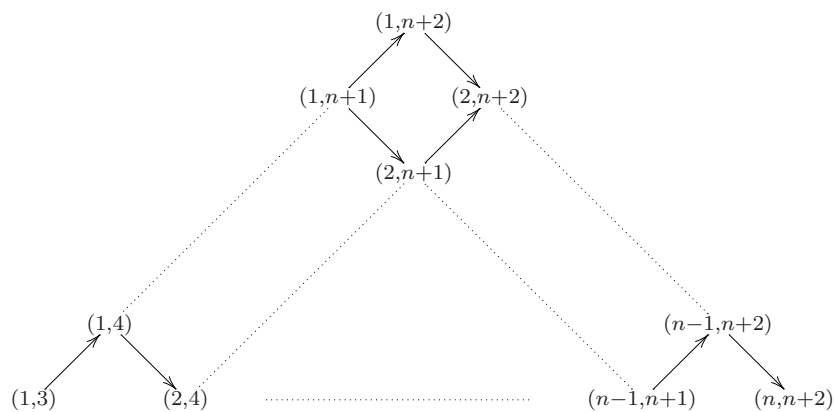
\[
\begin{xy}
(0,0) *{_{(1,3)}}="A" ,
(10,10) *{_{(1,4)}}="B" ,
(20,0) *{_{(2,4)}}="C" ,
(40,40) *{_{(1,n+1)}}="D" ,
(50,50) *{_{(1,n+2)}}="E" ,
(60,40) *{_{(2,n+2)}}="F" ,
(50,30) *{_{(2,n+1)}}="G" ,
(90,10) *{_{(n-1,n+2)}}="H" ,
(80,0) *{_{(n-1,n+1)}}="I" ,
(100,0) *{_{(n,n+2)}}="J" ,
(30,0) *{}="K" ,
(65,0) *{}="L" ,

\ar "A" ; "B"
\ar "B" ; "C" \ar@{.} "B" ; "D"
\ar@{.} "C" ; "G"
\ar "D" ; "E" \ar "D" ; "G"
\ar "E" ; "F"
\ar "G" ; "F" \ar@{.} "G" ; "I"
\ar@{.} "F" ; "H"
\ar "I" ; "H"
\ar "H" ; "J"
\ar@{.} "K" ; "L"
\end{xy}
\]
We remark that the vertex $(i,j)$ corresponds the $\utri{n}{K}$-module
\[
M_{ij}=\stackrel{\hspace{23pt} \stackrel{j-2}{\check{}} \hspace{12pt} \stackrel{1}{\check{}}}
{(\begin{smallmatrix} 0 & \cdots & 0 & K & \cdots & K \end{smallmatrix})} /
\stackrel{\hspace{17pt} \stackrel{i}{\check{}} \hspace{24pt} \stackrel{1}{\check{}}}
{(\begin{smallmatrix} 0 & \cdots & 0 & K & \cdots K \end{smallmatrix})}
= \stackrel{\hspace{23pt} \stackrel{j-2}{\check{}} \hspace{15pt} \stackrel{i}{\check{}} \hspace{23pt} \stackrel{1}{\check{}}}
{(\begin{smallmatrix} 0 & \cdots & 0 & K & \cdots & K & 0 & \cdots & 0 \end{smallmatrix})} .
\]

Next we consider a regular $(n+2)$-polygon $R_{n+2}$ whose vertices are numbered as follows.
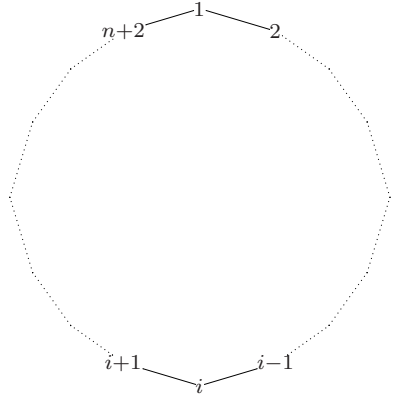
\[
\begin{xy}
(0,0) *{_1}="A" ,
(10,-3) *{_2}="B" ,
(17,-8) *{}="C" ,
(22,-15) *{}="CC" ,
(25,-25) *{}="CCC" ,
(22,-35) *{}="CCCC" ,
(17,-42) *{}="CCCCC" ,
(-10,-3) *{_{n+2}}="D" ,
(-17,-8) *{}="E" ,
(-22,-15) *{}="EE" ,
(-25,-25) *{}="EEE" ,
(-22,-35) *{}="EEEE" ,
(-17,-42) *{}="EEEEE" ,
(10,-47) *{_{i-1}}="F" ,
(0,-50) *{_i}="G" ,
(-10,-47) *{_{i+1}}="H" ,

\ar@{-} "A" ; "B"
\ar@{-} "D" ; "A"
\ar@{-} "F" ; "G"
\ar@{-} "G" ; "H"
\ar@{.} "B" ; "C"
\ar@{.} "C" ; "CC"
\ar@{.} "CC" ; "CCC"
\ar@{.} "CCC" ; "CCCC"
\ar@{.} "CCCC" ; "CCCCC"
\ar@{.} "CCCCC" ; "F"
\ar@{.} "D" ; "E"
\ar@{.} "E" ; "EE"
\ar@{.} "EE" ; "EEE"
\ar@{.} "EEE" ; "EEEE"
\ar@{.} "EEEE" ; "EEEEE"
\ar@{.} "EEEEE" ; "H"
\end{xy}
\]

We denote by  $D(R_{n+2})$ the set of all diagonals of $R_{n+2}$ except edges of $R_{n+2}$. 
We call a subset $S$ of $D(R_{n+2})$ a \emph{non-crossing partition}  of $R_{n+2}$ if $S$ satisfies the following conditions.
\begin{enumerate}
\def\labelenumi{(\theenumi)} 
\item Any two distinct diagonals in S do not cross except at their endpoints.
\item $R_{n+2}$ is divided into triangles by diagonals in $S$.
\end{enumerate}
We denote by $\mathcal{P}_{n+2}$ the set of an non-crossing partitions of $R_{n+2}$.

Now we construct the correspondence $\Phi$ from $\mathcal{P}_{n+2}$ to $\tilt{\utri{n}{K}}$. 
We take $S \in \mathcal{P}_{n+2}$.  
We remark that non-crossing partition of $R_{n+2}$ consists of $n-1$ diagonals. 
We denote by $(i,j)$ the diagonal between $i$ and $j$ for $i<j$ and put
\[
S = \{ (i_1,j_1), \ (i_2,j_2), \ \cdots, \ (i_{n-1},j_{n-1}) \}.
\]
Then we define
\[
\Phi(S):=M_{1,n+2} \oplus \left( \bigoplus_{k=1}^{n-1} M_{i_k,j_k} \right).
\]
It is shown that this is a basic tilting $\utri{n}{K}$-module.

\bigskip

Then the following hold.

\begin{thm}\label{tilt_utri}
The above correspondence $\Phi$ is a bijection.
\end{thm}
\begin{proof}
We divide the proof into five parts.

(i) One can easily check that the following conditions are equivalent for $(i,j) \neq (i',j')$.
\begin{itemize}
\item[(a)] $\ext{1}{R}{M_{i,j}}{M_{i',j'}}=0$ and $\ext{1}{R}{M_{i',j'}}{M_{i,j}}=0$.
\item[(b)] The diagonals $(i,j)$ and $(i',j')$ do not cross except at their endpoints.
\end{itemize}

(ii) For any $(i,j)$, we have that $\ext{1}{R}{M_{i,j}}{M_{i,j}}=0$.

(iii) $M_{1,n+2}$ is a projective injective $\utri{n}{K}$-module.

(iv) If $S$ is a non-crossing partition, then $\Phi(S)$ is a partial tilting module by (i),(ii) and (iii).
Since $\Phi(S)$ has $n$ non-isomophic indecomposable summands, it is a tilting module by Proposition \ref{equ_tilt}.

(v) Any basic tilting $\utri{n}{K}$-module $T$ has $M_{1,n+2}$ as a summand. 
Since $T$ has exactly $n$ indecomposable direct summands, there exists a subset $S$ of $D(R_{n+2})$ consists of $(n-1)$ elements 
such that $T=\Phi(S)$. By (i), $S$ is non-crossing partition of $R_{n+2}$.
\end{proof}

Theorem \ref{tilt_utri} gives a constructive bijection.

\begin{exam}
We consider $n=3$ case. We classify basic tilting $\utri{3}{K}$-modules by using Theorem \ref{tilt_utri}.
The partitions of the regular pentagon into triangles are given as follows.
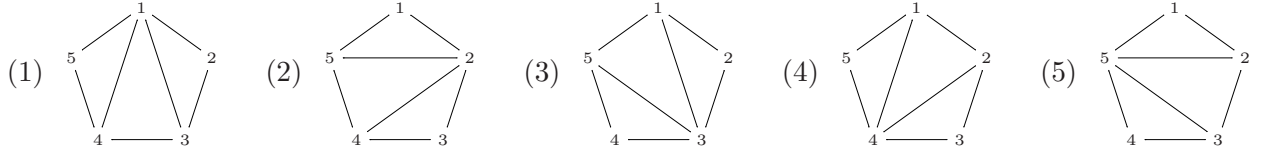
\[
\begin{array}{ccccc}
(1) \begin{array}{c}
\def\objectstyle{\scriptscriptstyle}
\def\alphanum{\ifcase\xypolynode\or 2 \or 1 \or 5 \or 4\or 3 
\fi}
 \xy \xygraph{!{/r2.3pc/:}
[] !P5"A"{~>{} ~*{\alphanum} }
}
\ar@{-} "A2";"A1"
\ar@{-} "A3";"A2"
\ar@{-} "A4";"A3"
\ar@{-} "A5";"A4"
\ar@{-} "A1";"A5"
\ar@{-}"A2";"A4"
\ar@{-}"A2";"A5"
\endxy
\end{array}&
(2) \begin{array}{c}
\def\objectstyle{\scriptscriptstyle}
\def\alphanum{\ifcase\xypolynode\or 2 \or 1 \or 5 \or 4\or 3 
\fi}
\xy \xygraph{!{/r2.3pc/:}
[] !P5"A"{~>{} ~*{\alphanum} }
}
\ar@{-} "A2";"A1"
\ar@{-} "A3";"A2"
\ar@{-} "A4";"A3"
\ar@{-} "A5";"A4"
\ar@{-} "A1";"A5"
\ar@{-}"A1";"A3"
\ar@{-}"A1";"A4"
\endxy
\end{array}&
(3) \begin{array}{c}
\def\objectstyle{\scriptscriptstyle}
\def\alphanum{\ifcase\xypolynode\or 2 \or 1 \or 5 \or 4\or 3 
\fi}
\xy \xygraph{!{/r2.3pc/:}
[] !P5"A"{~>{} ~*{\alphanum} }
}
\ar@{-} "A2";"A1"
\ar@{-} "A3";"A2"
\ar@{-} "A4";"A3"
\ar@{-} "A5";"A4"
\ar@{-} "A1";"A5"
\ar@{-}"A2";"A5"
\ar@{-}"A5";"A3"
\endxy
\end{array}&
(4) \begin{array}{c}
\def\objectstyle{\scriptscriptstyle}
\def\alphanum{\ifcase\xypolynode\or 2 \or 1 \or 5 \or 4\or 3 
\fi}
\xy \xygraph{!{/r2.3pc/:}
[] !P5"A"{~>{} ~*{\alphanum} }
}
\ar@{-} "A2";"A1"
\ar@{-} "A3";"A2"
\ar@{-} "A4";"A3"
\ar@{-} "A5";"A4"
\ar@{-} "A1";"A5"
\ar@{-}"A2";"A4"
\ar@{-}"A4";"A1"
\endxy
\end{array}&
(5) \begin{array}{c}
\def\objectstyle{\scriptscriptstyle}
\def\alphanum{\ifcase\xypolynode\or 2 \or 1 \or 5 \or 4\or 3 
\fi}
\xy \xygraph{!{/r2.3pc/:}
[] !P5"A"{~>{} ~*{\alphanum} }
}
\ar@{-} "A2";"A1"
\ar@{-} "A3";"A2"
\ar@{-} "A4";"A3"
\ar@{-} "A5";"A4"
\ar@{-} "A1";"A5"
\ar@{-}"A3";"A1"
\ar@{-}"A3";"A5"
\endxy
\end{array}
\end{array}
\]
Therefore the number of basic tilting $\utri{3}{K}$-modules is equal to $5$ and all of basic tilting $\utri{3}{K}$-modules are given as follows.
\begin{enumerate}
\def\labelenumi{(\theenumi)}
\item $(\begin{smallmatrix} K&K&K  \end{smallmatrix}) \oplus 
(\begin{smallmatrix} 0&K&K  \end{smallmatrix}) \oplus
(\begin{smallmatrix} 0&0&K  \end{smallmatrix})$, 
\item $(\begin{smallmatrix} K&K&K  \end{smallmatrix}) \oplus
(\begin{smallmatrix} K&K&0  \end{smallmatrix}) \oplus
(\begin{smallmatrix} 0&K&0  \end{smallmatrix})$, 
\item $(\begin{smallmatrix} K&K&K  \end{smallmatrix}) \oplus
(\begin{smallmatrix} K&0&0  \end{smallmatrix}) \oplus
(\begin{smallmatrix} 0&0&K  \end{smallmatrix})$, 
\item $(\begin{smallmatrix} K&K&K  \end{smallmatrix}) \oplus
(\begin{smallmatrix} 0&K&K  \end{smallmatrix}) \oplus
(\begin{smallmatrix} 0&K&0  \end{smallmatrix})$, 
\item $(\begin{smallmatrix} K&K&K  \end{smallmatrix}) \oplus
(\begin{smallmatrix} K&K&0  \end{smallmatrix}) \oplus
(\begin{smallmatrix} K&0&0  \end{smallmatrix})$.
\end{enumerate}
\end{exam}
\bigskip

Now we show examples of classifications of tilting modules over left  Harada algebras.

\begin{exam}  \ 
\begin{enumerate}
\def\labelenumi{(\theenumi)} 
\item Let $R$ be a local quasi-Frobenius algebra. Then we consider  block extension (c.f. \cite{Lec_N,OHr3})
\[
R(n) = \left( \begin{array}{ccc}
R & \cdots  & R \\
 & \ddots & \vdots \\
 \jac{R} & & R
\end{array} \right)
\]
for $n \in \mathbb{N}$ of $R$ which is a subalgebra of $n \times n$ full matrix algebra over $R$. 
We can show that
\begin{itemize}
\item[(a)] the first row is a injective module,
\item[(b)] the $i$-th row is the Jacobson radical of the $(i-1)$-th row for $2 \leq i \leq n$.
\end{itemize}
In particular $R(n)$ is a left Harada algebra  with $m=1$ and $n_1=n$ in Definition \ref{def_Harada}. 

By Corollary \ref{main_cor}, we have a bijection $F: \tilt{R(n)} \longrightarrow \tilt{\utri{n}{K}}$.
We can obtain all basic tilting $R(n)$-modules from the definition of $F$ and Theorem \ref{tilt_utri}.
\item Let $R$ be a basic quasi-Frobenius algebra which has complete set of orthogonal primitive idempotents $\{e,f\}$. 
Then we can represent $R$ as follows (c.f. \cite{AF}).
\[
R \simeq \left( \begin{array}{ccc}
eRe & eRf \\
fRe & fRf 
\end{array} \right).
\]
We put $Q:=eRe$, $W:=fRf$, $A:=eRf$ and $B:=fRe$. 
Now we consider the block extension (c.f. \cite{Lec_N,OHr3})
\[
R(n_1,n_2)=\left( \begin{array}{ccc|ccc}
Q & \cdots  & Q & A & \cdots &A \\
 & \ddots & \vdots  & \vdots & & \vdots \\
 \jac{Q} & & Q & A& \cdots & A \\ \hline
B &\cdots & B & W & \cdots  & W \\
\vdots & & \vdots & & \ddots & \vdots \\
B & \cdots & B & \jac{W} & & W
\end{array} \right)
\]
for $n_1,n_2 \in \mathbb{N}$ of $R$ which is a subalgebra of $\enmo{R}{(eR)^{n_1} \oplus (fR)^{n_2}}$.
We can show that
\begin{itemize}
\item[(a)] the first and $n_1+1$ row are injective modules,
\item[(b)] the $i$-th row is the Jacobson radical of the $(i-1)$-th row for $2 \leq i \leq n$ and $n+2 \leq i \leq n+m$.
\end{itemize}
In particular $R(n_1,n_2)$ is a left Harada algebra with $m=2$ in Definition \ref{def_Harada}. 

By Corollary \ref{main_cor}, we have a bijection $F: \tilt{R(n_1,n_2)} \longrightarrow \tilt{\utri{n_1}{K}} \times \tilt{\utri{n_2}{K}}$.
We can obtain all basic tilting $R(n_1,n_2)$-modules from the definition of $F$ and Theorem \ref{tilt_utri}.
\end{enumerate}
\end{exam}

\end{document}